\documentclass[10pt]{amsart}
\usepackage{amssymb,amsbsy,amsmath,amsfonts,amssymb,amscd}
\usepackage{latexsym,euscript,exscale}

\title{Finitely approximable groups and actions\\
Part II: Generic representations}
\author {Christian Rosendal}
\address{Department of Mathematics, Statistics, and Computer Science (M/C 249)\\
University of Illinois at Chicago\\
851 S. Morgan St.\\
Chicago, IL 60607-7045\\
USA}
\email{rosendal.math@gmail.com}
\urladdr{http://www.math.uic.edu/$~$rosendal}

\newcommand {\A}{\mathbf A}

\newcommand {\F}{\mathbb F}

\newcommand {\N}{\mathbb N}
\newcommand {\Q}{\mathbb Q}

\newcommand {\Z}{\mathbb Z}

\renewcommand{\leq}{\ensuremath{\leqslant}}
\renewcommand{\geq}{\ensuremath{\geqslant}}

\newcommand{\tom} {\emptyset}

\newcommand{\saa}{\Rightarrow}
\newcommand{\equi}{\Leftrightarrow}

\newcommand{\til}{\rightarrow}

\newcommand{\Lim}[1]{\mathop{\longrightarrow}\limits_{#1}}

\newcommand {\del}{ \; \big| \;}

\newcommand {\ku} {\mathcal}

\newcommand{\inv}{^{-1}}

\newcommand {\e} {\exists}
\renewcommand {\a} {\forall}

\newtheorem{thm}{Theorem}
\newtheorem{cor}[thm]{Corollary}
\newtheorem{lemme}[thm]{Lemma}

\newtheorem{prop} [thm] {Proposition}
\newtheorem{defi} [thm] {Definition}

\usepackage{fouriernc}

\begin{document}

\thanks{The author was partially supported by NSF grants DMS 0901405 and DMS 0919700. The author is also grateful  for the many helpful suggestions of the referee.}



\maketitle

\begin{abstract}
Given a finitely generated group $\Gamma$, we study the space ${\rm Isom}(\Gamma,{\mathbb Q\mathbb U})$ of all 
actions of $\Gamma$ by isometries of the rational Urysohn metric space ${\mathbb Q\mathbb U}$, where ${\rm Isom}(\Gamma,{\mathbb Q\mathbb U})$ is equipped
with the topology it inherits seen as a closed subset of ${\rm Isom}({\mathbb Q\mathbb U})^\Gamma$. When $\Gamma$ is
the free group $\F_n$ on $n$ generators this space is just ${\rm Isom}({\mathbb Q\mathbb U})^n$, but is in general significantly more complicated. We
prove that when $\Gamma$ is finitely generated Abelian  there is a generic
point in ${\rm Isom}(\Gamma,{\mathbb Q\mathbb U})$, i.e., there is a comeagre set of
mutually conjugate isometric actions of $\Gamma$ on ${\mathbb Q\mathbb U}$. 
\end{abstract}

\section{Introduction}
\subsection{Representations of discrete groups in topological groups}\label{rep}
Suppose $\Gamma$ is a discrete group and $G$ a Hausdorff topological
group. A {\em representation} of $\Gamma$ in $G$ is simply a
group homomorphism $\pi\colon \Gamma\til G$. We shall depending on the context use the notations $\pi(g)$ and $g^\pi$ for the image of $g\in \Gamma$ by the homomorphism $\pi$.
Since a representation is a function from $\Gamma$ to $G$, it is formally an element of $G^\Gamma$. Moreover, the set of all representations of $\Gamma$ in $G$ is a closed subset 
of $G^\Gamma$, namely, 
$$
{\rm Rep}(\Gamma,G)=\{\pi\in G^\Gamma\del 
\a g,f \in \Gamma\;\;\;\; \pi(g)\pi(f)=\pi(gf)\}.
$$

The conjugacy action of $G$ on itself extends naturally to a {\em diagonal conjugacy action} of $G$ on $G^\Gamma$ by letting $G$  act separately on each coordinate, and one easily sees that ${\rm Rep}(\Gamma,G)$ is $G$-invariant. Thus, two representations $\pi$ and $\theta$ are {\em conjugate} if there is an element $a\in G$ such that
$$
a\pi(g)a\inv=\theta(g)
$$
for all $g\in \Gamma$.

Now, when $S\subseteq \Gamma$ is a generating set (finite or infinite), any representation $\pi\in {\rm Rep}(\Gamma,G)$ is fully specified by the restriction $\pi|_S\in G^S$ and so the restriction map  
$$
\begin{cases}
\;\;\;{\rm Rep}(\Gamma,G)\til G^S\\
\;\;\;\pi\mapsto \pi|_S
\end{cases}
$$ 
is injective. Moreover, the image of ${\rm Rep}(\Gamma,G)$ in $G^S$ is closed, for to see if some $\sigma\colon S\til G$ extends to a homomorphism $\tilde\sigma\colon \Gamma\til G$, it suffices by von Dyck's Theorem (see Theorem 5.8 \cite{bogopolski}) to check that $\sigma(s_1)\cdots \sigma(s_n)=1$ whenever $s_1\cdots s_n=1$ for $s_i\in S$, which is easily seen to be a closed condition in $G^S$. Also, the inverse map is continuous, since for every $g=s_1\cdots s_n\in \Gamma$, the coordinate map, $\sigma\mapsto \sigma(s_1)\cdots \sigma(s_n)\in G$, is continuous.
Finally, the above map is evidently $G$-equivariant for the diagonal conjugacy actions on respectively $G^\Gamma$ and $G^S$. So, up to a $G$-equivariant homeomorphism,  we may identify ${\rm Rep}(\Gamma,G)$ with a closed $G$-subspace of $G^S$. 

Of course, when $\Gamma$ is finitely generated, we shall choose $S\subseteq \Gamma$ finite. In this way, ${\rm Rep}(\F_n,G)$, where $\F_n$ is the free group on $n$-generators, is naturally identified with $G^n$ and ${\rm Rep}(\Z^n,G)$ is naturally identified with the set of commuting $n$-tuples $(g_1,\ldots, g_n)\in G^n$, i.e., such that $g_ig_j=g_jg_i$ for all $i,i$.

In the following, we shall solely be concerned with the case where
$G$ is Polish, i.e., separable and completely metrisable. Moreover,
$G$ will in fact be the isometry group of a countable metric space, namely the rational Urysohn metric space ${\mathbb Q\mathbb U}$, which we shall define later. Here $G={\rm Isom}({\mathbb Q\mathbb U})$ will be given the {\em permutation group topology}, whose basic open sets are of the form 
$$
U(f;x_1,\ldots, x_n)=\{g\in {\rm Isom}({{\mathbb Q\mathbb U}})\del g(x_i)=f(x_i),\; i\leqslant n\}.
$$
Equivalently, a neighbourhood basis at the identity is given by the sets
$$
G_A=\{g\in {\rm Isom}({{\mathbb Q\mathbb U}})\del \a x\in A\; g(x)=x\},
$$
where $A$ varies over finite subsets of ${\mathbb Q\mathbb U}$.
In this case, representations of $\Gamma$ in ${\rm Isom}({\mathbb Q\mathbb U})$ are just actions of $\Gamma$ on
${\mathbb Q\mathbb U}$ by isometries, and we shall therefore denote the space of
representations by ${\rm Isom}(\Gamma,{\mathbb Q\mathbb U})$ instead of the more
cumbersome ${\rm Rep}(\Gamma,{\rm Isom}({\mathbb Q\mathbb U}))$. Notice that since
${\rm Isom}({\mathbb Q\mathbb U})$ is Polish and $\Gamma$ countable, ${\rm Isom}(\Gamma,{\mathbb Q\mathbb U})$
is a closed subset of the Polish space ${\rm Isom}({\mathbb Q\mathbb U})^\Gamma$ and hence ${\rm Isom}(\Gamma,{\mathbb Q\mathbb U})$ is a Polish space in itself
on which ${\rm Isom}({\mathbb Q\mathbb U})$ is acting continuously.

If $S\subseteq \Gamma$ is a fixed finite generating set of a finitely generated group $\Gamma$, a basic open neighbourhood of a representation $\pi\in {\rm Isom}(\Gamma, {{\mathbb Q\mathbb U}})$ is given by
$$
U(\pi,A)=\{\sigma\in {\rm Isom}(\Gamma,{{\mathbb Q\mathbb U}})\del g^\pi(x)=g^\sigma(x), \; x\in A\;\&\; g\in S\},
$$
where $A\subseteq {\mathbb Q\mathbb U}$ is any finite subset.

Now, if instead $\Gamma$ fails to be finitely generated, we need also to specify the set $S\subseteq \Gamma$. So the basic open neighbourhoods of $\pi\in {\rm Isom}(\Gamma,{{\mathbb Q\mathbb U}})$ are of the form 
$$
U(\pi,A,S)=\{\sigma\in {\rm Isom}(\Gamma,{{\mathbb Q\mathbb U}})\del g^\pi(x)=g^\sigma(x), \; x\in A\;\&\; g\in S\},
$$
where $A\subseteq {\mathbb Q\mathbb U}$ and $S\subseteq  \Gamma$ are any finite subsets.

In any case, we see that if $\pi\in U(\sigma,A)$, resp. $\pi\in U(\sigma, A,S)$, then $U(\pi,A)=U(\sigma,A)$, resp. $U(\pi,A,S)=U(\sigma,A,S)$.

\begin{defi}
Let $\Gamma$ be a countable discrete group and  $G$ a Polish group. We say that a representation $\pi$ of $\Gamma$ in $G$ is
\begin{itemize}
\item  {\em generic} if the $G$-orbit of $\pi$, $G\cdot \pi$, is comeagre in ${\rm Rep}(\Gamma,G)$.
\item {\em locally generic} if the $G$-orbit of $\pi$ is non-meagre in
${\rm Rep}(\Gamma,G)$.
\item {\em dense} if the $G$-orbit of $\pi$ is dense in ${\rm Rep}(\Gamma,G)$.
\end{itemize}

\end{defi}
The set of dense representations is easily seen to be a $G_\delta$
set and hence if non-empty it is dense $G_\delta$. Thus, if there is
a dense and a locally generic representation, then the locally
generic representation must also be dense and hence generic.

The existence of dense and generic
representations of the groups $\Z$ and $\F_n$ in various  Polish groups has been
extensively studied in the literature (see, e.g., 
\cite{akin,glasner,herwig,ivanov,book, turbulence,lascar,truss}
and the references therein). In the literature on ergodic theory the
existence of a dense representation of $\Z$ in a Polish group $G$,
which is of course just the existence of a dense conjugacy class in
$G$, is sometimes denoted by saying that $G$ has the topological
Rokhlin property, as the proof of this in the case of $G={\rm
Aut}([0,1],\lambda)$ relies on Rokhlin's lemma. Also in the
literature on model theory, the existence of generic representations
of $\F_n$ for all $n$ in a Polish group $G$ is denoted by saying
that $G$ has {\em ample generics}. The import of ample generics or even a comeagre conjugacy class for
the structure theory of $G$ is considerable as can be sampled from
\cite{hodges,turbulence, macpherson}.


\subsection{The rational Urysohn metric space}
The {\em Urysohn metric space} $\mathbb U$ is a universal separable metric space initially constructed by P. Urysohn \cite{urysohn}, which is fully characterised up to isometry by being separable and complete, together with the  following extension property.
\begin{quotation}
If $\phi\colon X\til \mathbb U$ is an isometric embedding of a
finite metric space $X$ into $\mathbb U$ and $Y=X\cup \{y\}$ is a one
point metric extension of $X$, then $\phi$ extends to an isometric
embedding of $Y$ into $\mathbb U$.
\end{quotation}
There is also a rational variant of $\mathbb U$ called the {\em rational Urysohn metric space}, which we denote by ${\mathbb Q\mathbb U}$. This is, up to isometry, the unique countable metric space with only rational distances such that the following variant of the above extension property holds.
\begin{quotation}
If $\phi\colon X\til {\mathbb Q\mathbb U}$ is an isometric embedding of a
finite  metric space $X$ into ${\mathbb Q\mathbb U}$ and $Y=X\cup \{y\}$ is a one
point metric extension of $X$ whose metric only takes rational distances, then $\phi$ extends to an isometric
embedding of $Y$ into ${\mathbb Q\mathbb U}$.
\end{quotation}

An isometry $f\colon A\til B$ between finite subsets $A$ and $B$ of the rational Urysohn space ${\mathbb Q\mathbb U}$ is said to be a {\em finite partial isometry} of ${\mathbb Q\mathbb U}$. So the restriction of any full isometry of ${\mathbb Q\mathbb U}$, that is, an isometry of ${\mathbb Q\mathbb U}$ onto itself,  to a finite subset is a finite partial isometry. But more importantly, by a back and forth argument, any finite partial isometry of ${\mathbb Q\mathbb U}$ extends to a full isometry of ${\mathbb Q\mathbb U}$, in other words, ${\mathbb Q\mathbb U}$ is {\em ultrahomogeneous}.

But much more is true. Namely, we have the following fact due to V. V. Uspenski\u\i{} \cite{uspenskii}: If $\Gamma$ is a group acting by isometries on a finite subspace $A\subseteq {\mathbb Q\mathbb U}$, then the action of $\Gamma$ extends to an action by isometries on all of ${\mathbb Q\mathbb U}$. To see this, we can without loss of generality suppose that $\Gamma$ is finite. Also, modulo an inductive construction, it suffices to show that for any one-point extension $B\supseteq A$, there is a further finite extension $C\supseteq B$ and an action of $\Gamma$ on $C$ extending the action of $\Gamma$ on $A$.  We identify the unique point in $B\setminus A$ with $1\in \Gamma$. We can now extend the metric $d$ on $B=A\sqcup\{1\}$ to all of $C=A\sqcup \Gamma$, by letting $d(a,h)=d(h\inv a, 1)$ for $a\in A$ and $h\in \Gamma$, and setting
$$
d(g,h)=\min \big(d(a,g)+d(a,h)\del a\in A\big).
$$
This is easily seen to be a metric extending the metric on $B$ and the invariance under the left-shift action by $\Gamma$ is trivial.

For actions of a group $\Gamma$ on infinite subspaces of ${\mathbb Q\mathbb U}$, we must change the conclusion somewhat. For this, it will be useful to introduce some terminology.
If $\pi\colon \Gamma\curvearrowright X$ and $\sigma\colon \Gamma\curvearrowright Y$ are actions of a group
$\Gamma$ on sets $X\subseteq Y$, we say that $\pi$ is a {\em subaction} or {\em subrepresentation} of $\sigma$, denoted by $\pi\leqslant \sigma$, if for any $x\in X$ and $g\in \Gamma$, $g^\pi(x)=g^\sigma(x)$. So, in particular, for this to hold, $X$ must be invariant under the action $\sigma$.
Also, if $\Delta\leqslant \Gamma$ and $\pi\colon \Gamma\curvearrowright X$ is an action of $\Gamma$ on a set $X$, we let $\pi|_\Delta$ denote the corresponding action of $\Delta$ on $X$.

A $\Gamma$-map between actions $\pi\colon \Gamma\curvearrowright X$ and $\sigma\colon \Gamma\curvearrowright Y$  is a function $\iota\colon X\til Y$ such that for any $x\in X$ and $g\in \Gamma$, $\iota(g^\pi(x))=g^\sigma(\iota(x))$.  We also say that $\iota$ {\em  conjugates} $\pi$ with $\sigma$. Note that we do not require $\iota$ to be surjective.

Using essentially the above construction and Kat\u etov's construction of the Urysohn space, Uspenski\u\i{} proved the following result, which we state for the rational Urysohn space.

\begin{thm}(V. V. Uspenski\u\i{} \cite{uspenskii})\label{uspenskii}
Suppose $\pi\colon \Gamma\curvearrowright (X,d_X)$ is an action of a  group $\Gamma$ by isometries on a rational metric space $(X,d_X)$. Then there is an   action $\sigma\colon \Gamma\curvearrowright {{\mathbb Q\mathbb U}}$ by isometries and an isometric injection $\iota \colon (X,d_X)\til {\mathbb Q\mathbb U}$ conjugating $\pi$ with $\sigma$.
\end{thm}


\subsection{The profinite topology and finitely approximable groups}
The {\em profinite topology} on a group $\Gamma$ is the topology generated by cosets $gK$ of finite index normal subgroups $K\leqslant \Gamma$. Thus, a subset $S\subseteq \Gamma$ is closed in the profinite topology on $\Gamma$ if for any $g\in \Gamma\setminus S$, there is a finite index normal subgroup $K\leqslant \Gamma$ such that $g\notin SK$. Since this is a group topology, $\Gamma$ is Hausdorff if and only if $\{1\}$ is closed, i.e., if for any $g\neq 1$ there is a finite index subgroup $K$ not containing $g$. In other words, $\Gamma$ is Hausdorff if and only if it is residually finite.

A group $\Gamma$ is {\em subgroup separable} if any finitely generated subgroup $H\leqslant \Gamma$ is closed in the profinite topology on $\Gamma$. So, as $\{1\}$ is finitely generated, subgroup separability implies residual finiteness. M. Hall \cite{hall1,hall2} originally proved that free groups are subgroup separable. On the other hand, e.g., $\F_2\times \F_2$ is residually finite, but not subgroup separable \cite{mihailova}.

However, the even stronger notion of relevance to us is the  Ribes--Zalesski\u\i{} property, or property (RZ) for brevity. Here a group $\Gamma$ is said to have the {\em  Ribes--Zalesski\u\i{} property} if any product $H_1H_2\cdots H_n$ of finitely generated subgroups $H_i\leqslant \Gamma$ is closed in the profinite topology on $\Gamma$. This property was originally proven for free groups by L. Ribes and P.  A. Zalessski\u\i{} in \cite{ribes} and T. Coulbois \cite{coulbois} showed that if both $\Gamma$ and $\Lambda$ have property (RZ), then so does $\Gamma*\Lambda$.

We note that if $H_1, \ldots, H_n$ are finitely generated subgroups of an Abelian group $\Gamma$, then $H=H_1\cdots H_n$ is again a finitely generated subgroup of $\Gamma$. So for Abelian groups, subgroup separability and property (RZ) coincides. It is an easy exercise to show that finitely generated Abelian groups are subgroup separable, which essentially follows from them being residually finite.


\subsection{Results}
The starting point of our study is a long series of investigations by model theorists into which automorphism groups of countable structures have comeagre conjugacy classes or ample generics. Often, but not always, the existence of a comeagre conjugacy class or ample generics is proved by verifying a specific combinatorial property of closing off finite partial automorphisms, see, e.g., \cite{herwig, hodges}.

For the specific case of the rational Urysohn metric space, this property is due to S. Solecki, whose proof in turn relied on work by B. Herwig and D. Lascar \cite{herwig}.
\begin{thm}[S. Solecki \cite{solecki}]\label{solecki}
Let $A$ be a finite rational metric space. Then there is a finite rational metric space $B$ containing $A$ and such that any partial isometry of $A$ extends to a full isometry of $B$.
\end{thm}

From this, we have the following corollary (see \cite{turbulence,solecki}) .
\begin{cor}
For any $n\geq 1$, the free group $\F_n$ has a generic representation in ${\rm Isom}({\mathbb Q\mathbb U})$, i.e., ${\rm Isom}({\mathbb Q\mathbb U})$ has ample generics.
\end{cor}

The main goal of the present paper is to establish the same conclusion for a much larger class of finitely generated groups, namely those satisfying property (RZ). For this, we will rely on the explicit analysis of Solecki's Theorem \ref{solecki} from our companion paper \cite{approx}, which shows that the groups all of whose actions on ${\mathbb Q\mathbb U}$ are {\em finitely approximable} are exactly those with property (RZ).
\begin{thm}
Let $\Gamma$ be a finitely generated group with property (RZ). Then  there is a
generic representation $\pi$ in ${\rm Isom}(\Gamma,{\mathbb Q\mathbb U})$.
Moreover, every orbit of $\Gamma$ under the action $\pi$ on ${\mathbb Q\mathbb U}$
is finite.
\end{thm}

So, in particular, this applies to finitely generated Abelian groups.

For groups that are not finitely generated, the situation is in general somewhat different. However, one can make generic representations cohere over an increasing approximating chain of subgroups, which leads to the following result.

\begin{thm}Let $\Gamma$ be a countable group that is an increasing union of finitely generated groups with property (RZ). Then there is a
representation $\pi$ of $\Gamma$ on ${\mathbb Q\mathbb U}$ such that for all
finitely generated subgroups $\Delta\leq\Gamma$ the representation $\pi|_\Delta$ is generic.
\end{thm}

Again, this applies, for example, to the additive group of rational numbers $\Q$.

For some background reading on Polish groups and their actions, we can refer the reader to Kechris' book \cite{kechris}, and for more information on the current research on the Urysohn metric space, the special volume of Topology and its Applications \cite{top} is a good place to start.


\section{Existence of dense and generic representations}
Suppose that $\Gamma$ is a countable group. We wish to characterise when $\Gamma$ has a generic representation on ${\mathbb Q\mathbb U}$. The reason not to study the the case of dense representations is that, as we shall see now, any $\Gamma$ admits a dense representation. 

\begin{prop}
Let $\Gamma$ be a countable group. Then $\Gamma$ has a dense representation on ${\mathbb Q\mathbb U}$.
\end{prop}

\begin{proof}
To show that $\Gamma$ admits a dense representation on ${\mathbb Q\mathbb U}$, by the Baire Category Theorem, it suffices to show that the action of $G={\rm Isom}({\mathbb Q\mathbb U})$\ on ${\rm Isom}(\Gamma,{\mathbb Q\mathbb U})$ is  topologically transitive, i.e., that for any two non-empty open sets $V,W\subseteq {\rm Isom}(\Gamma, {\mathbb Q\mathbb U})$ there is some $g\in G$ such that $g\cdot V\cap W\neq \tom$. 

So let $U(\pi,A,S)$ and $U(\sigma,B,T)$ be basic open neighbourhoods of actions $\pi,\sigma\in {\rm Isom}(\Gamma, {\mathbb Q\mathbb U})$.
We define the expanded value set
$$
{\rm Ex}(A)=\{r_1+\ldots+r_n\del r_i\in d_X[A\times A]\;\&\; r_1+\ldots+r_n\leqslant  {\rm diam}(A)\},
$$
and recall from \cite{approx} that the following defines a $\pi$-invariant metric on ${\mathbb Q\mathbb U}$ that agrees with $d$ on the subset $A$
$$
\partial_A(x,y)=
\begin{cases}
\min (s\in {\rm Ex}(A)\del d(x,y)\leqslant s)        &\text{if $d(x,y)\leqslant {\rm diam}(A)$,}\\
{\rm diam}(A)                                                   &\text{otherwise.}
\end{cases}
$$
In the same manner, we can define the expanded value set ${\rm Ex}(B)$ and a $\sigma$-invariant metric $\partial_B$ on ${\mathbb Q\mathbb U}$. We now let $(X,\partial)$ be the disjoint union of the metric spaces $({{\mathbb Q\mathbb U}},\partial_A)$ and $({{\mathbb Q\mathbb U}},\partial_B)$, 
$$
(X,\partial)=({{\mathbb Q\mathbb U}},\partial_A)\oplus({{\mathbb Q\mathbb U}},\partial_B),
$$
where for $x$ belonging to $({{\mathbb Q\mathbb U}},\partial_A)$ and $y$ belonging to $({{\mathbb Q\mathbb U}},\partial_B)$, we set
$$
\partial(x,y)=\max\big\{{\rm diam}(A), {\rm diam}(B)\big\}.
$$
Note that $\Gamma$ acts by isometries of $(X,\partial)$ by acting via $\pi$ on $({{\mathbb Q\mathbb U}},\partial_A)$ and via $\sigma$ on $({{\mathbb Q\mathbb U}},\partial_B)$. Denote this action by $\pi\oplus\sigma$. By Theorem \ref{uspenskii}, it follows that there is an isometric embedding $\iota\colon (X,\partial)\til ({{\mathbb Q\mathbb U}},d)$ such that for some isometric action $\tau\colon \Gamma\curvearrowright ({{\mathbb Q\mathbb U}},d)$, $\iota$ conjugates the action $\pi\sqcup\sigma$  with $\tau$. Thus, by ultrahomogeneity of ${\mathbb Q\mathbb U}$, there are $g,f\in G$ such that  $g\cdot \tau\in U(\pi,A,S)$ and $f\cdot \tau\in U(\sigma,B,T)$, whence $fg\inv \cdot U(\pi,A,S)\cap U(\sigma,B,T)\neq\tom$.
\end{proof}

Now, to characterise the existence of generic representations, we adapt the results of \cite{ivanov, turbulence,truss} that in various generalities treated the case of representations of $\F_n$ by automorphisms of ultrahomogeneous first order structures. The main difference is that ultrahomogeneity is of relatively little use when considering actions of general countable or finitely generated groups and we must therefore content ourselves with a less finitary characterisation.

We first need the following lemma from \cite{turbulence}. 
\begin{lemme}\label{turbulence} Suppose $G$ is a Polish group acting continuously on a Polish space $X$ and let $x\in X$.  Then
the following are equivalent:
\begin{enumerate}
\item For every neighbourhood of the identity $V\subseteq G$, $V\cdot x$ is comeagre in a
neighbourhood of $x$.
\item For each neighbourhood of the identity $V\subseteq G$, $V\cdot x$ is somewhere dense.
\item The orbit $G\cdot x$ is non-meagre.
\end{enumerate}
\end{lemme}

 \begin{proof}(1)$\saa$(3) is trivial. Also, for (3)$\saa$(2), suppose
$G\cdot x$ is non-meagre and $V\subseteq G$ is a neighbourhood of $1$.  Then
we can find $g_n\in G$ such that $G=\bigcup_n g_nV$, whence $G\cdot x=\bigcup_ng_nV\cdot x$. So some $g_nV\cdot x$, and therefore also $V\cdot x$, is non-meagre and hence somewhere dense.

Finally, for (2)$\saa$(1), suppose that $V\cdot x$ is somewhere dense for
every neighbourhood $V\subseteq G$ of $1$. Suppose towards a contradiction that for some neighbourhood $U\subseteq G$ of $1$, $U\cdot x$ is meagre, whence there are  
closed nowhere dense sets $F_n\subseteq X$ covering $U\cdot x$. But then the sets
$K_n=\{g\in G\del g\cdot x\in F_n\}$ are closed and cover $U$ and thus,
by the Baire category theorem, some $K_n$ contains a non-empty open
set $gV$, where $V$ is a neighbourhood of $1$ and $g\in G$. So
$gV\cdot x\subseteq F_n$ and $V\cdot x$ must be nowhere dense, which is a
contradiction. 

Now, if $V\subseteq G$ is any neighbourhood of $1$, let $U\subseteq V$ be a smaller neighbourhood such that $U\inv U\subseteq V$. Then $U\cdot x$ is comeagre in some neighbourhood of a point $g\cdot x$, where $g\in U$, and thus $g\inv U\cdot x\subseteq V\cdot x$ is comeagre in a neighbourhood of $x$.
\end{proof}

\begin{lemme}\label{comeagre crit}
Suppose $G$ is a Polish group acting continuously on a Polish space $X$. Then the following are equivalent:
\begin{enumerate}
\item There is a non-meagre orbit $\ku O\subseteq X$.
\item There is a non-empty open set $O\subseteq X$ with the following property:
For all open $\tom\neq V\subseteq O$ and neighbourhood $U\subseteq G$ of $1$, there is a smaller open $\tom\neq W  \subseteq V$ such that the action of $U$ on $W$ is topologically transitive, i.e., for any non-empty open $W_0,W_1\subseteq W$ there is $g\in U$ such that $gW_0\cap W_1\neq \tom$.
\end{enumerate}
Moreover, if $\ku O$ is an orbit comeagre in an open set $O\subseteq X$, then (2) holds for $O$.
\end{lemme}

\begin{proof}
(1)$\saa$(2): If $\ku O\subseteq X$ is a non-meagre orbit, let $O\subseteq X$ be a non-empty open set in which $\ku O$ is comeagre. Now, if $V\subseteq O$ is non-empty open and $U\subseteq G$ is a neighbourhood of $1$, pick $x\in V\cap \ku O$ and choose an open neighbourhood  $U_0\subseteq U$ of $1$ such that $U_0U_0\inv\subseteq U$.
Then, by the preceding lemma, $U_0\cdot x$ is dense in some open neighbourhood $W\subseteq V$ of $x$ and it follows that the action of $U$ on $W$ is topologically transitive.

(2)$\saa$(1): Suppose $O\subseteq X$ is an open set satisfying the assumption in (2). Fix a neighbourhood basis $\{U_n\}_{n\in \N}$ at $1\in G$ and a basis $\{V_n\}_{n\in \N}$ for the induced topology on $O$ consisting of non-empty open sets. Now, for every $n$ and $m$, let $W_{n,m}\subseteq V_n$ be a non-empty open subset such that the action of $U_m\inv $ on $W_{n,m}$ is topologically transitive. Then $W_m=\bigcup_nW_{n,m}$ is open dense in $O$ since it intersects every $V_n$. Also, for any $V_k\subseteq W_{n,m}$, $W_{n,m}\cap \big(U_m\inv \cdot V_k\big)$ is open dense in $W_{n,m}$, and so 
$$
D_{n,m}=W_{n,m}\cap \bigcap_{V_k\subseteq W_{n,m}} \big(U_m\inv \cdot V_k\big)
$$
is comeagre in $W_{n,m}$. Note also that if $x\in D_{n,m}$, then for any $V_k\subseteq W_{n,m}$, $U_m\cdot x\cap V_k\neq \tom$, showing that $U_m\cdot x$ is  dense in $W_{n,m}$.  We notice that $D_m=\bigcup_nD_{n,m}$ is comeagre in $O$ and that for any $x\in D_m$, $U_m\cdot x$ is somewhere dense. It follows that for any $x$ belonging to the comeagre subset $\bigcap_mD_m\subseteq O$, and for  any $k$, $U_k\cdot x$ is somewhere dense, which by the previous lemma implies that $G\cdot x$ is non-meagre. \end{proof}

Suppose $\Gamma$ is a countable group and set  $G={\rm Isom}({\mathbb Q\mathbb U})$. Since the action of $G$ on ${\rm Isom}(\Gamma,{\mathbb Q\mathbb U})$ has a dense orbit, there is in fact a comeagre set of $\pi \in {\rm Isom}(\Gamma,{\mathbb Q\mathbb U})$ having dense orbits, whence any locally generic $\pi\in {\rm Isom}(\Gamma,{\mathbb Q\mathbb U})$ will be generic. 
So, by Lemma \ref{comeagre crit}, we see that $\Gamma$ admits a generic representation on ${\mathbb Q\mathbb U}$ if and only if the following condition holds:
\begin{quote}
For all finite $A\subseteq {\mathbb Q\mathbb U}$, $R\subseteq \Gamma$ and $\rho\in {\rm Isom}(\Gamma,{\mathbb Q\mathbb U})$, there are finite $A\subseteq B\subseteq {\mathbb Q\mathbb U}$, $R\subseteq S\subseteq \Gamma$ and some $\sigma\in U(\rho,A,R)$ such that  for all finite $B\subseteq C\subseteq {\mathbb Q\mathbb U}$, $S\subseteq T\subseteq \Gamma$ and $\tau,\pi \in U(\sigma,B,S)$ 
$$
G_A\cdot U(\tau,C,T)\cap U(\pi,C,T)\neq \tom.
$$
\end{quote}
Of course, if $\Gamma$ is finitely generated with a fixed finite generating set $S\subseteq \Gamma$ (which is not specified in the notation below), the above criterion simplifies to the following.
\begin{quote}
For all finite $A\subseteq {\mathbb Q\mathbb U}$ and $\rho\in {\rm Isom}(\Gamma,{\mathbb Q\mathbb U})$, there is a finite $A\subseteq B\subseteq {\mathbb Q\mathbb U}$ and some $\sigma\in U(\rho,A)$ such that  for all finite $B\subseteq C\subseteq {\mathbb Q\mathbb U}$ and $\tau,\pi \in U(\sigma,B)$
$$
G_A\cdot U(\tau,C)\cap U(\pi,C)\neq \tom.
$$
\end{quote}

In the following, if $\pi\colon \Gamma\til G$ is a representation and $D\subseteq \Gamma$ is a subset, we set $D^\pi=\{g^\pi\in G\del g\in D\}$. In particular, $\Gamma^\pi$ is a subgroup of $G$.

\begin{lemme}\label{extending finite orbits}
Suppose $\Gamma$ is a finitely generated group with a fixed finite generating set $S\subseteq \Gamma$.  Let $\sigma\in {\rm Isom}(\Gamma,{\mathbb Q\mathbb U})$  and suppose that $B\subseteq {\mathbb Q\mathbb U}$ is a finite  $\Gamma^\sigma$-invariant subset. Then for all $\tau,\pi\in U(\sigma, B)$ and finite $B\subseteq C\subseteq {\mathbb Q\mathbb U}$, 
$$
G_B\cdot U(\tau,C)\cap U(\pi, C)\neq \tom.
$$
\end{lemme}

\begin{proof}
Note that, since $\tau,\pi\in U(\sigma, B)$,  for every generator $g\in S$ and every $x\in B$, we have $g^\tau(x)=g^\pi(x)=g^\sigma(x)\in B$. So it follows that $B$ is invariant under both $\tau$ and $\pi$. We set $X=\Gamma^\tau\cdot C$ and $Y=\Gamma^\pi\cdot C$ and define a pseudometric $\partial$ on the disjoint union $X\sqcup Y$ by letting $\partial$ agree with the metric $d$ on $X$ and $Y$ separately and for $x\in X$ and $y\in Y$ setting
$$
\partial(x,y)=\min_{z\in B}\big (d(x,z)+d(z,y)\big).
$$
We denote by $X\sqcup_\partial Y$ the metric space obtained from $X\sqcup Y$ by identifying points of distance $0$. Thus, $X\sqcup_\partial Y$ is obtained by freely amalgamating $X$ and $Y$ over the common subspace $B$ and we can therefore see $X$ and $Y$ as subspaces of $X\sqcup_\partial Y$ that intersect exactly in their common copy of $B$. We now let $\rho$ be the action of $\Gamma$ on $X\sqcup_\partial Y$ defined by setting 
$$
g^\rho(x)=\begin{cases}
g^\sigma(x)& \text{ if $x\in B$;}\\
g^\tau(x)& \text{ if $x\in X$;}\\ 
g^\pi(x)& \text{ if $x\in Y$.}
\end{cases}
$$
Then $\rho$ is easily seen to be an action by isometries extending the action $\tau$ on $X$ and the action $\pi$ on $Y$. By Theorem \ref{uspenskii}, there is an isometric embedding $\iota\colon X\sqcup_\partial Y\til {\mathbb Q\mathbb U}$ and an isometric action $\rho_0\colon \Gamma\curvearrowright {\mathbb Q\mathbb U}$ such that $\iota$ conjugates $\rho$ with $\rho_0$. Conjugating with an element of $G$, we can suppose that $\iota$ is the identity on $B$.   Now, since $S\subseteq \Gamma$ is finite, so are $C\cup S^\tau\cdot C\subseteq X$ and $C\cup S^\pi\cdot C\subseteq Y$, and thus, by ultrahomogeneity of ${\mathbb Q\mathbb U}$, we can find isometries $f,h\in G_B$ such that $f\cdot \rho_0\in U(\tau,C)$ and $h\cdot \rho_0\in U(\pi,C)$, whence $G_B\cdot U(\tau,C)\cap U(\pi, C)\neq \tom$.
\end{proof}

\begin{thm}\label{generic}
Let $\Gamma$ be a finitely generated group with property (RZ). Then $\Gamma$ has a generic representation in ${\rm Isom}({\mathbb Q\mathbb U})$.
\end{thm}

\begin{proof}
By Lemma \ref{extending finite orbits}, to verify the criterion for existence of generic representations, it suffices that for every $\rho\in {\rm Isom}(\Gamma,{\mathbb Q\mathbb U})$ and finite $A\subseteq {\mathbb Q\mathbb U}$, there is $\sigma\in U(\rho,A)$ and a finite $\Gamma^\sigma$-invariant subset $A\subseteq B\subseteq {\mathbb Q\mathbb U}$. Let $S$ be the fixed set of generators of $\Gamma$. Now, by the main result of \cite{approx}, given $A$  and $\rho$, there is is a finite rational metric space $Y$ containing $A \cup S^\rho\cdot A$ and an isometric action $\pi\colon \Gamma\curvearrowright Y$ such that for all $g\in S$  and all $x\in A$, $g^\pi(x)=g^\rho(x)$. Now, by Theorem \ref{uspenskii} and the ultrahomogeneity of ${\mathbb Q\mathbb U}$, we can suppose that actually $A\cup S^\rho\cdot A\subseteq Y\subseteq {\mathbb Q\mathbb U}$ and that the action $\pi$ extends to an action $\sigma\colon \Gamma\curvearrowright {\mathbb Q\mathbb U}$. Letting $B=Y$, we have the result.
\end{proof}

We note that the proof above establishes the following important property, which we shall be using again. For any $\rho\in {\rm Isom}(\Gamma,{\mathbb Q\mathbb U})$ and finite $A\subseteq {\mathbb Q\mathbb U}$, there is $\sigma\in U(\rho,A)$ with a finite $\Gamma^\sigma$-invariant subset $A\subseteq B\subseteq {\mathbb Q\mathbb U}$.

Conjugacy  is possibly the finest notion of similarity that can be imposed on representations of a countable group $\Gamma$ in a topological group $G$. Thus, the existence of generic representations is similarly a very strong requirement that seldom holds outside of automorphism groups of first order structures. We shall now consider a much coarser notion, which can be expected to hold more generally. 

 Note that, in general, $\Gamma^\pi$ is not closed in $G$, but is still a topological group in the induced  topology from $G$.
\begin{defi}
Two representations $\pi$ and $\tau$ of a countable group $\Gamma$ in a topological group 
$G$ are said to be {\em topologically similar} if ${\rm ker}\;\pi={\rm ker}\;\tau$ and the map
$$
g^\pi\in \Gamma^\pi\mapsto g^\tau\in \Gamma^\tau
$$
is an isomorphism of topological groups.
\end{defi}
Now, since a homomorphism between topological groups is continuous if it is continuous at the identity, we see that two faithful representations $\pi, \tau$ of $\Gamma$ in $G$ are topologically similar if and only if for any net $(g_i)$ in $\Gamma$, 
$$
g_i^\pi\Lim{i\til \infty}1\;\equi\; g^\tau_i\Lim{i\til \infty}1.
$$
So if $G$ is metrisable, let $\ku B$ be a countable neighbourhood basis at the identity. Then topological similarity of $\pi$ and $\tau$ is given by
$$
\a U\in \ku B\;\; \e V\in \ku B\;\;\a g\in \Gamma\; \;\big[\big ( g^\tau\in \overline U\text{ or }g^\pi\notin  V\big )\;\&\; \big ( g^\pi\in \overline U\text{ or }g^\tau\notin  V\big)\big],
$$
showing that topological similarity is an $F_{\sigma\delta}$ equivalence relation.
Obviously, the conjugacy relation refines topological similarity.

Letting $\F_\infty$ denote the free group on countably many generators $a_1,a_2,\ldots$, we have the following result, which indicates that one should expect few generic representations of non-finitely generated countable groups. 
\begin{prop}
Let $G$ be a non-trivial Polish group. Then topological similarity classes in ${\rm Rep}(\F_\infty,G)$ are meagre.
\end{prop}

\begin{proof}
Notice first that ${\rm Rep}(\F_\infty,G)$ can be identified with $G^\N$ by sending any $\pi\in {\rm Rep}(\F_\infty,G)$ to the sequence $(a_n^\pi)\in G^\N$. So if $\pi=(g_n)\in G^\N$ and $\sigma=(f_n)\in G^\N$ are topologically similar, then for any sequence $(n_k)$ of natural numbers, 
$$
g_{n_k}\Lim{k\til \infty}1\;\equi\; f_{n_k}\Lim{k\til \infty}1.
$$
Now, for any infinite set $S\subseteq \N$, the set
$$
\A(S)=\{(g_n)\in G^\N\del \e (n_k)\subseteq S\;\; g_{n_k}\Lim{k\til \infty}1\}
=\bigcap_m\big\{(g_n)\in G^\N\del \e s\in S\; d(g_s,1)<\frac 1m\big\} 
$$
is dense $G_\delta$ in $G^\N$ and is invariant under topological similarity. Thus, if $C\subseteq G^\N$ were a non-meagre topological similarity class, we would have 
$$
C\subseteq \bigcap_{S\subseteq \N\text{ infinite}} \A(S)=\big\{(g_n)\in G^\N\del g_n\Lim{n\til\infty}1\big\},
$$
contradicting that $\big\{(g_n)\in G^\N\del g_n\Lim{n\til\infty}1\big\}$ is meagre in $G^\N$.
\end{proof}


\section{Coherence properties of generic representations}
Suppose $G$ is a Polish group and $\Gamma$ is a countable group
generated by two subgroups $\Delta$ and $\Lambda$. Then the mapping $\pi\mapsto (\pi|_\Delta,\pi|_\Lambda)$ identifies 
${\rm Rep}(\Gamma,G)$ with a subset of
${\rm Rep}(\Delta,G)\times{\rm Rep}(\Lambda,G)$,
and, by the reasoning of Section \ref{rep}, the image of
${\rm Rep}(\Gamma,G)$ is closed in
${\rm Rep}(\Delta,G)\times{\rm Rep}(\Lambda,G)$. 

The following result shows that, though
${\rm Rep}(\Gamma,G)$ is not in general equal to the product
space ${\rm Rep}(\Delta,G)\times{\rm Rep}(\Lambda,G)$, we
still have a version of the  Kuratowski--Ulam Theorem (see  (8.41) in \cite{kechris}). We recall that if $A\subseteq Y\times Z$ is a subset of a product space, we let $A_y=\{z\in Z\del (y,z)\in
A\}$.  Also, if $P$ is a property of points in a Polish space $X$, we write $\a^* x\in X \;\; P(x)$ if $P$ holds on a comeagre set of $x\in X$.

\begin{thm}\label{cor coherence}
Suppose $G$ is a Polish group and $\Gamma$ a countable group
generated by two subgroups $\Delta$ and  $\Lambda$. 
Suppose there is a generic $(\rho_0,\sigma_0)\in {\rm Rep}(\Gamma,G)$
such that $\rho_0$ is generic in ${\rm Rep}(\Delta,G)$. Then
$$
\a^* \rho\in {\rm Rep}(\Delta,G)\;\a^* \sigma\in{\rm Rep}(\Gamma,G)_\rho\quad(\rho,\sigma) \textrm{ is
generic in  } {\rm Rep}(\Gamma,G).
$$
\end{thm}

As pointed out by the referee, the original proof of Theorem \ref{cor coherence} amounted to a proof of the following probably well-known variation of the Kuratowski--Ulam Theorem. However, since we do not know of a reference, we include the short proof here.

\begin{thm}\label{kurulam}
Suppose $f\colon X\til Y$ is a surjective, continuous, open function between Polish spaces $X$ and $Y$. Then for any set $A\subseteq X$ with the Baire property, the following are equivalent
\begin{enumerate}
\item $A$ is comeagre,
\item $\a ^* y\in Y\;  \;\; A\cap f\inv(y) \text{ is comeagre in }f\inv (y)$.
\end{enumerate}
\end{thm}

\begin{proof}
To see the implication from (1) to (2), fix a  basis $\{U_n\}_{n\in \N}$ for the topology on $X$ and find dense open sets $D_n\subseteq X$ such that $\bigcap_{n\in \N}D_n\subseteq A$. Then, for any $y\in Y$, 
\[\begin{split}
\big(\bigcap_{n\in \N}D_n\big)\cap f\inv (y) &\text{ is comeagre  in }f\inv (y)\\
&\equi \a n\;\; D_n\cap f\inv (y) \text{ is dense in }f\inv (y)\\
&\equi \a n\;\a m\; \big( f\inv(y)\cap U_m\neq \tom \til  f\inv(y)\cap U_m\cap D_n\neq \tom\big)\\
&\equi  \a n\;\a m\;\;  y\notin f( U_m)\setminus  f(U_m\cap D_n)\\
&\equi y\in \bigcap_{n,m\in \N} \big(  f(U_m\cap D_n)\cup \sim \!\!f( U_m)\big).
\end{split}\]
Now, since $f$ is surjective open, $f(U_m\cap D_n)\cup \sim \!\!f( U_m)$ is dense open in $Y$ for all $n,m$, whence $\bigcap_{n,m\in \N} \big(  f(U_m\cap D_n)\cup \sim\!\! f( U_m)\big)$ is comeagre in $Y$, showing that (1) implies (2).

To see that (2) implies (1), suppose that $A\subseteq X$ is not comeagre in $X$ and find a non-empty open set $V\subseteq X$ such that $V\setminus A$ is comeagre in $V$. Applying (1)$\saa$(2) to the mapping $f\colon V\til f(V)$, we see that 
$$
\a ^* y\in f(V)\;  \;\; \big(V\setminus A\big)\cap f\inv(y)\cap V \text{ is comeagre in }f\inv (y)\cap V,
$$
whence 
$$
\e ^* y\in Y\;  \;\; A\cap f\inv(y) \text{ is not comeagre in }f\inv (y),
$$
which finishes the proof.
\end{proof}

And now we can present the proof of Theorem \ref{cor coherence}.
\begin{proof}
By a result of D. Marker and R. L. Sami (see \cite{becker}), non-meagre orbits are necessarily $G_\delta$. So it follows that both the spaces $X=\big(G\cdot \rho_0\times {\rm Rep}(\Lambda,G)\big)\cap {\rm Rep}(\Gamma,G)$ and $Y=G\cdot  \rho_0$ are Polish. We claim that the $G$-equivariant projection map $\pi(\rho,\sigma)=\rho$ from $X$ to $Y$ is open (it is clearly surjective and continuous). To see this, note that if $U\subseteq X$ is open, then $V=\{g\in G\del \e \sigma\in {\rm Rep}(\Lambda,G)\;\;\; (g\cdot \rho_0,\sigma)\in U\}$ is open too, whence, by Effros' Theorem (see Theorem 2.2.2 in \cite{becker}), $\pi(U)=V\cdot \rho_0$ is open in $Y$.

Now, letting $A=G\cdot (\rho_0,\sigma_0)$, which is comeagre in $X$, we have by Theorem \ref{kurulam} that
$$
\a^* \rho \in Y\;\;\; A\cap \pi\inv(\rho)\text{ is comeagre in }\pi\inv(\rho), 
$$
i.e.,
$$
\a^* \rho \in Y\;\;\;\a^* \sigma \in {\rm Rep}(\Gamma,G)_\rho \;\;\;(\rho,\sigma)\in G\cdot (\rho_0,\sigma_0). 
$$
Since $Y=G\cdot  \rho_0$ is comeagre in ${\rm Rep}(\Delta,G)$, it follows that
$$
\a^* \rho \in {\rm Rep}(\Delta,G)\;\;\;\a^* \sigma \in {\rm Rep}(\Gamma,G)_\rho \;\;\;(\rho,\sigma)\text{ is generic in }{\rm Rep}(\Gamma,G), 
$$
which finishes the proof.
\end{proof}


\section{More on generic representations}
We shall now see how the coherence properties of Theorem \ref{cor coherence} play out in the case of representations in $G={\rm Isom}({\mathbb Q\mathbb U})$. 
\begin{lemme}\label{coherence}
Suppose $\Gamma$ is a group and $\Lambda\leqslant \Gamma$ is a subgroup. Assume that $\pi \colon \Gamma\curvearrowright X$ is an action by isometries on a metric space $X$ and $\sigma\colon \Lambda\curvearrowright Y$ is an action by isometries on a metric space containing $X$ such that $\sigma$ extends the action $\pi|_\Lambda$. Then there is a metric space $Z$ containing $Y$ and an action by isometries $\tau\colon \Gamma\curvearrowright Z$ such that $\tau$ extends $\pi$ and $\tau|_\Lambda$ extends $\sigma$.

Moreover, if $Y$ is a rational metric space and $X$ is finite, then $Z$ can be made a rational metric space.
\end{lemme} 

\begin{proof}
Let $d$ denote the metric on $Y$ and define the following pseudometric $\partial$ on $Y\times \Gamma$.
$$
\partial \big((y_1,g_1),(y_2,g_2)\big)=\begin{cases}
d(g_2\inv g_1\cdot y_1,y_2), &\text{ if $y_1,y_2\in X$ or $g_2\inv g_1\in \Lambda$;}\\
\inf_{x\in X}d(y_1,g_1\inv\cdot x)+d(g_2\inv \cdot x,y_2), &\text{ otherwise.}
\end{cases}
$$
By considering cases and using the easy fact that for all $(y_1,g_1), (y_2,g_2)\in Y\times \Gamma$, 
$$
\partial \big((y_1,g_1),(y_2,g_2)\big)\leqslant \inf_{x\in X}d(y_1,g_1\inv\cdot x)+d(g_2\inv \cdot x,y_2),
$$
one checks that $\partial$ indeed is a pseudometric. 

We now let $\tau\colon \Gamma\curvearrowright Y\times \Gamma$ be the action by left-translation on the second coordinate. This clearly preserves $\partial$. Let now $\sim$ be the equivalence relation on $Y\times \Gamma$ given by 
$$
(y_1,g_1)\sim(y_2,g_2)\;\equi\; \partial\big( (y_1,g_1),(y_2,g_2)\big)=0
$$
and let $[y,g]$ denote the equivalence class of $(y,g)$. Then $\partial$ defines a metric on $Y\times \Gamma/\sim$ and the action $\tau$ factors through to an isometric action of $\Gamma$ on $Y\times \Gamma/\sim$. Also, the map $\iota\colon Y\til Y\times \Gamma/\sim$ defined by $\iota(y)=[y,1]$ is an isometric embedding. 

So to prove the result, it suffices to show that $\iota$ conjugates $\sigma$ with $\tau|_\Lambda$ and that $\iota|_X$ conjugates $\pi$ with $\tau$, whence, by renaming, we can let $Z=Y\times \Gamma/\sim$.

To see this, suppose that either $y\in Y$ and $g\in \Lambda$ or that $y\in X$ and $g\in \Gamma$. Then
$$
\iota(g\cdot y)=[g\cdot y,1]=[y, g]=g\cdot [y,1]=g\cdot \iota(y),
$$
since $\partial\big((g\cdot y,1),(y,g)\big)=d(g\inv1 g\cdot y,y)=0$.

For the moreover part, note that if $d$ takes rational values and $X$ is finite, then also $\partial$ takes rational values.
\end{proof}

\begin{prop}\label{coherence prop}
Suppose $\Gamma$ is a finitely generated group with property (RZ) and $\Lambda\leqslant \Gamma$ is a finitely generated subgroup. Then for any generic representation $\pi\in {\rm Isom}(\Gamma,{\mathbb Q\mathbb U})$, also $\pi|_\Lambda\in {\rm Isom}(\Lambda,{\mathbb Q\mathbb U})$ is generic. It follows that any generic $\sigma\in {\rm Isom}(\Lambda,{\mathbb Q\mathbb U})$ is the restriction of a generic $\pi\in {\rm Isom}(\Gamma,{\mathbb Q\mathbb U})$ to $\Lambda$, i.e., $\sigma=\pi|_\Lambda$.
\end{prop}

\begin{proof}
Fix finite generating sets $S\subseteq T$ for $\Lambda$ and $\Gamma$ respectively. 
Let $p\colon {\rm Isom}(\Gamma,{\mathbb Q\mathbb U})\til {\rm Isom}(\Lambda,{\mathbb Q\mathbb U})$ denote the projection defined by $p(\pi)=\pi|_\Lambda$ and let $V_n\subseteq {\rm Isom}(\Lambda,{\mathbb Q\mathbb U})$  be a sequence of dense open sets  whose intersection is the set of generic representations of $\Lambda$ in ${\rm Isom}({\mathbb Q\mathbb U})$. For the first part, it suffices to show that $p\inv(V_n)$ is dense open in ${\rm Isom}(\Gamma,{\mathbb Q\mathbb U})$, since then $\bigcap_n p\inv (V_n)$ will be comeagre and hence contain a generic.

So suppose $U\subseteq {\rm Isom}(\Gamma,{\mathbb Q\mathbb U})$ is any non-empty open set. By the proof of Theorem \ref{generic}, we can find $\pi\in U$ and a finite $\Gamma^\pi$-invariant subset $X\subseteq {\mathbb Q\mathbb U}$ such that $U(\pi,X)\subseteq U$. Since $V_n$ is dense open in ${\rm Isom}(\Lambda,{\mathbb Q\mathbb U})$, there is a finite set $X\subseteq B\subseteq {\mathbb Q\mathbb U}$ and some $\sigma\in U(\pi|_\Lambda,X)$ such that $U(\sigma, B)\subseteq V_n$. Set $Y=\Lambda^{\sigma}\cdot B\subseteq {\mathbb Q\mathbb U}$. By Lemma \ref{coherence}, we can find a countable rational metric space $Z$ containing $Y$ and an action by isometries $\tau\colon \Gamma\curvearrowright Z$ such that $\tau$ extends the action $\pi\colon \Gamma\curvearrowright X$ and $\tau|_\Lambda$ extends the action $\sigma\colon \Lambda\curvearrowright Y$. Now, by Theorem \ref{uspenskii} and the ultrahomogeneity of ${\mathbb Q\mathbb U}$, there is an isometric injection $\iota\colon Z\til {\mathbb Q\mathbb U}$, which is the identity on $B\cup S^{\tau|_\Lambda}\cdot B$ and an action $\rho\in {\rm Isom}(\Gamma,{\mathbb Q\mathbb U})$ such that $\iota$ conjugates $\tau$ with $\rho$. We thus see that 
$$
p(\rho)=\rho|_\Lambda\in U(\sigma,B)\subseteq V_n,
$$
while
$$
\rho\in U(\pi,X)\subseteq U.
$$
It follows that $p\inv (V_n)\cap U\neq \tom$, showing that $p\inv (V_n)$ is dense open.

Now, for the second part, suppose $\sigma\in {\rm Isom}(\Lambda,{\mathbb Q\mathbb U})$ is generic. Let $\tau\in {\rm Isom}(\Gamma,{\mathbb Q\mathbb U})$ be any generic, whereby also $\tau|_\Lambda$ is generic. It follows that for some $g\in {\rm Isom}({\mathbb Q\mathbb U})$, $\sigma=g\cdot (\tau|_\Lambda)=(g\cdot \tau)|_\Lambda$. Since also $g\cdot \tau$ is generic, we see that $\sigma$ is the restriction of a generic $\pi=g\cdot\tau\in {\rm Isom}(\Gamma, {\mathbb Q\mathbb U})$ to $\Lambda$.
\end{proof}
Under the assumptions of Proposition \ref{coherence prop}, we see that if $\ku O\subseteq {\rm Isom}(\Lambda,{\mathbb Q\mathbb U})$ is the set of generic representations of $\Lambda$ and $\ku C\subseteq {\rm Isom}(\Gamma,{\mathbb Q\mathbb U})$ is the set of generic representations of $\Gamma$, then $\ku O=\ku C|_\Lambda$.

Assume $\Gamma$ is a group with property (RZ), generated by finitely generated subgroups $\Lambda$ and $\Delta$. By Proposition \ref{coherence prop}, if $\pi\colon \Gamma\curvearrowright {\mathbb Q\mathbb U}$ is generic, then also $\pi|_\Lambda$ and $\pi|_\Delta$ are generic. But conversely, by Theorem \ref{cor coherence}, we have the following. Suppose $\rho\colon \Lambda\curvearrowright {\mathbb Q\mathbb U}$ is a generic representation and let $X_\rho\subseteq {\rm Isom}(\Delta,{\mathbb Q\mathbb U})$ denote the closed set of $\sigma\colon \Delta\curvearrowright {\mathbb Q\mathbb U}$ such that $\rho$ and $\sigma$ are restrictions of the same $\pi\colon \Gamma\curvearrowright {\mathbb Q\mathbb U}$. Let also $\ku O\subseteq{\rm Isom}(\Delta,{\mathbb Q\mathbb U})$ denote the set of generic representations of $\Delta$. Then $\ku O\cap X_\rho$ is comeagre in $X_\rho$.

The case of finitely generated free Abelian groups $\Z^n$ is of special interest to us. First recall that $\Z^n$ has property (RZ). Also, the space of representations of $\Z^n$ by isometries on ${\mathbb Q\mathbb U}$ can be identified with the set of commuting $n$-tuples in ${\rm Isom}({\mathbb Q\mathbb U})$, i.e.,
$$
{\rm Isom}(\Z^n,{{\mathbb Q\mathbb U}})=\{(g_1,\ldots,g_n)\in {\rm Isom}({{\mathbb Q\mathbb U}})\del g_ig_j=g_jg_i\text{ for }i,j\leqslant n\}.
$$
Denoting by $C(g_1,\ldots,g_n)$ the centraliser of $\{g_1,\ldots, g_n\}$, we see that for $(g_1,\ldots,g_n)\in {\rm Isom}(\Z^n,{{\mathbb Q\mathbb U}})$ and $f\in {\rm Isom}({\mathbb Q\mathbb U})$, we have 
$$
(g_1,\ldots, g_n, f)\in {\rm Isom}(\Z^{n+1},{{\mathbb Q\mathbb U}})\;\;\equi\;\; f\in C(g_1,\ldots,g_n).
$$
Thus, for this special case, our results imply the following.

\begin{cor}
For every finite number $n$ there is a generic commuting $n$-tuple
in ${\rm Isom}({\mathbb Q\mathbb U})$. Moreover, for all generic
commuting $n$-tuples $(g_1,\ldots, g_n)$, there is a comeagre set of
$f\in C(g_1,\ldots,g_n)$ such that $(g_1,\ldots, g_n,f)$ is a
generic commuting $n+1$-tuple. In particular, there is a comeagre
conjugacy class in $C(g_1,\ldots,g_n)$.
\end{cor}

\begin{proof}
Only the very last statement is non-trivial. So let $\ku O\subseteq
C(g_1,\ldots,g_n)$ be the set of $h$ such that $(g_1,\ldots, g_n,h)$
is a generic commuting $n+1$-tuple. Then for all $h,f\in \ku O$
there is some $k\in{\rm Isom}({\mathbb Q\mathbb U})$ such that
$$
(g_1,\ldots, g_n,h)=(kg_1k\inv,\ldots, kg_nk\inv,kfk\inv).
$$
But then $k$ commutes with each of $g_i$ and hence belongs to
$C(g_1,\ldots,g_n)$. So $h$ and $f$ are conjugate by an element of
$C(g_1,\ldots,g_n)$.
\end{proof}

We should mention a curious phenomenon, namely that if
$(g_1,\ldots, g_n)$ is a generic commuting $n$-tuple, then there is
some $k$ such that
$$
kg_1k\inv=g_2,\quad kg_2k\inv=g_3,\;\ldots\;, kg_nk\inv =g_1,
$$
whence, in particular, $k^ng_ik^{-n}=g_i$ for all $i$.
To see this, notice that if $\sigma$ is a permutation of
$\{1,\ldots,n\}$, then 
$$
(h_1,\ldots,h_n)\mapsto (h_{\sigma(1)},\ldots, h_{\sigma(n)})
$$
is a $G$-equivariant homeomorphism of the space of commuting $n$-tuples in $G$ with itself and so maps the comeagre orbit onto itself. In particular, $(g_1,\ldots, g_n)$ and $(g_2,g_3,\ldots, g_n, g_1)$ are diagonally conjugate by some $k\in G$.

\begin{thm}\label{RZ groups} 
Let $\Gamma$ be a countable group that is the increasing union of a chain of finitely generated RZ groups, e.g., if $\Gamma$ itself has property RZ. Then there is a
representation $\pi$ of $\Gamma$ on ${\mathbb Q\mathbb U}$ such that for all
finitely generated subgroups $\Lambda\leqslant \Gamma$ the representation $\pi|_\Lambda$ is generic.
\end{thm}

\begin{proof}
Write $\Gamma$ as a union of a chain of
finitely generated RZ subgroups
$$
\Delta_0\leq \Delta_1\leq \Delta_2\leq\ldots\leq \Gamma.
$$
Then, by Proposition \ref{coherence prop}, we can inductively define generic $\pi_n\in {\rm Isom}(\Delta_n,{\mathbb Q\mathbb U})$ such that $\pi_n=\pi_{n+1}|_{\Delta_n}$ for all $n$.
Seeing the $\pi_n$ as homomorphisms from $\Delta_n$ to ${\rm Isom}({\mathbb Q\mathbb U})$, $\bigcup_n\pi_n$ naturally defines a representation $\pi\in {\rm Isom}(\Gamma,{\mathbb Q\mathbb U})$.
To see that $\pi$ is as requested, suppose $\Lambda\leqslant \Gamma$ is any finitely generated subgroup and find $n$ such that $\Lambda\leqslant \Delta_n$. Then $\pi|_{\Delta_n}=\pi_n$ is generic, and so also $\pi|_\Lambda=\pi_n|_{\Lambda}$ is generic by Proposition \ref{coherence prop}.
\end{proof}

Now, since $\Z$ has a generic representation in ${\rm Isom}({\mathbb Q\mathbb U})$, in particular, ${\rm Isom}({\mathbb Q\mathbb U})$ has a comeagre conjugacy class. We shall refer to the elements of this conjugacy class as the {\em generic} elements of ${\rm Isom}({\mathbb Q\mathbb U})$.
The following result was already obtained in \cite{powers} by other means.

\begin{cor}
The generic isometry of ${\mathbb Q\mathbb U}$ has roots of all orders.
Moreover, if $f$ is generic and $n\neq 0$, then $f$ is conjugate
with $f^n$.
\end{cor}

\begin{proof}
Let $\pi$ be a representation of $\Q$ as given by Theorem
\ref{RZ groups} and fix $n\neq 0$. Then $\pi$ also induces generic representations of the infinite cyclic subgroups  $\Delta=\langle 1\rangle$ and $\Lambda=\langle n\rangle$, and hence $\pi|_\Delta$ and $\pi|_\Lambda$ are therefore conjugate representations. That is, if 
$h=1^\pi$ and $g={n}^\pi=(1+1+\ldots +1)^\pi=h^n$, then $h$ and $g$ are generic and thus congugate in ${\rm Isom}({\mathbb Q\mathbb U})$.

Now, suppose $f$ is any generic element of ${\rm Isom}({\mathbb Q\mathbb U})$. Then there is some $k\in {\rm Isom}({\mathbb Q\mathbb U})$ such that $f=kgk\inv=kh^nk\inv=(khk\inv)^n$, showing that $f$ has an $n$th root, which moreover is generic.

Also, as $h$ is generic, there is some $l$ such that $f=lhl\inv$, whence
$f^n=lh^nl\inv=lgl\inv=lk\inv\cdot kgk\inv\cdot kl\inv=lk\inv\cdot f\cdot kl\inv$, showing that $f$ and $f^n$ are conjugate.
\end{proof}

It is clear that our results hold for a somewhat larger class of metric spaces, namely, for the Fra\"iss\'e limits of finite metric spaces corresponding to a restrictive class of countable distance sets. However, in order not to complicate notation and assumptions, we have chosen to present only the case of $\mathbb Q\mathbb U$, which already contains the ideas for the general case. Let us just  mention that with only minor changes in proofs, we can replace $\mathbb Q\mathbb U$ with the Urysohn metric spaces with distance set $\{0,1,\ldots, n\}$ for any finite $n$. So, e.g., the case $n=1$ corresponds to the case of a countable discrete set and $n=2$ to the case of the random graph.



\begin{thebibliography}{999}



\bibitem{akin}  E. Akin, M. Hurley and J. A. Kennedy, {\it Dynamics of
Topologically Generic Homeomorphisms}, Memoirs of Amer. Math.
Soc., {\bf 164}, No. 783, 2003.



\bibitem{becker} H. Becker and A. S. Kechris, {\it The Descriptive Set
Theory of Polish Group Actions}, London Math. Soc. Lecture Note
Series, {\bf 232}, Cambridge Univ. Press, 1996.

\bibitem{bogopolski} O. Bogopolski, {\em Introduction to group theory},  
Translated, revised and expanded from the 2002 Russian original. EMS Textbooks in Mathematics. European Mathematical Society (EMS), Z\"urich, 2008. 



\bibitem{coulbois}T. Coulbois, {\em Free product, profinite topology and finitely generated subgroups}, International Journal of Algebra and Computation, Vol. 11, No. 2 (2001), pp. 171--184.





\bibitem{glasner}E. Glasner and B. Weiss; {\em Topological groups with Rokhlin properties},
Colloq. Math. 110 (2008), no. 1, 51--80.




\bibitem{hall1}M. Hall, Jr., {\em Coset Representations in Free Groups}, Transactions of the American Mathematical Society, Vol. 67, No. 2 (Nov., 1949), 
421--432.

\bibitem{hall2} M. Hall, Jr., {\em A Topology for Free Groups and Related Groups}, Annals of Mathematics, Second Series, Vol. 52, No. 1 (Jul., 1950), 127--139.



\bibitem{herwig} B. Herwig and D. Lascar, Extending partial automorphisms
and the profinite topology on free groups, {\it Trans. Amer. Math.
Soc.}, {\bf 352}, no. 5, 1985--2021, 2000.


\bibitem{hodges} W. Hodges, I. Hodkinson, D. Lascar, and S. Shelah, {\em  The
small index property for $\omega$-stable $\omega$-categorical
structures and for the random graph},  J. London Math. Soc.,
{\bf 48 (2)}, 204--218, 1993.


\bibitem{ivanov} A. A. Ivanov, Generic expansions of $\omega$-categorical
structures and semantics of generalized quantifiers, {\it J.
Symbolic Logic}, {\bf 64 (2)}, 775--789, 1999.

\bibitem{kechris} A. S. Kechris; Classical descriptive set theory, Spinger Verlag, New York, 1995.



\bibitem{book}A. S. Kechris; {\em Global aspects of ergodic group actions}, Mathematical Surveys and Monographs, 160, American Mathematical Society, 2010.

\bibitem{turbulence}A. S. Kechris and  C. Rosendal; Turbulence, amalgamation, and generic automorphisms of homogeneous structures, {\em Proc. London Math. Soc., 94 (2007) no.2, 302--350}.



\bibitem{lascar} D. Lascar, Les beaux automorphismes, {\it
Arch. Math. Logic}., {\bf 31 (1)}, 55--68, 1991.


\bibitem{macpherson} D. Macpherson and S. Thomas, Comeagre conjugacy classes
and free products with amalgamation,  {\it  Discrete Math.}  291  (2005),  no. 1-3, 135--142.



\bibitem{mihailova}K. A. Miha\u\i lova, {\em 
The occurrence problem for direct products of groups}, Mat. Sb. (N.S.) 70 (112) 1966 241--251. 








\bibitem{ribes}L. Ribes and P. A. Zalesski\u\i, {\em On the profinite topology on a free group}. Bull. London 
Math. Soc. 25, 37--43 (1993)




\bibitem{powers}C. Rosendal, {\em The generic isometry and measure preserving homeomorphism are conjugate to their powers}, Fundamenta Mathematicae 205, No. 1, 1-27 (2009).



\bibitem{approx}C. Rosendal, {\em Finitely approximable groups and actions. Part I: The Ribes--Zalesski\u\i{} property}, preprint.




\bibitem{solecki}S. Solecki, {\em Extending partial isometries}, Israel J. Math. 150 (2005), 315--332.



\bibitem{truss} J. K. Truss, Generic automorphisms of homogeneous
structures, {\it Proc. London Math. Soc.}, {\bf 65 (3)}, 121--141,
1992.




\bibitem{urysohn}P. Urysohn, {\em Sur un espace m\'etrique universel},
Bull. Sci. Math. 51 (1927), 43-64, 74-90.


\bibitem{uspenskii} V. V. Uspenski\u\i{}, {\em On the group of isometries of the Urysohn universal metric space},  
Comment. Math. Univ. Carolin. 31 (1990), no. 1, 181--182. 


\bibitem{top}Topology and its Applications, Volume 155, Issue 14, Pages 1451-1634 (15 August 2008) 
Special Issue: Workshop on the Urysohn space
Ben-Gurion University of the Negev, Beer Sheva, ISRAEL
21-24 May 2006
Edited by A. Leiderman, V. Pestov, M. Rubin, S. Solecki and V. Uspenskij


\end{thebibliography}
\end{document}